\documentclass[12pt]{article}
\usepackage{a4}
\usepackage{amsthm}
\usepackage{amsmath}
\allowdisplaybreaks
\usepackage{amssymb}
\usepackage{amsfonts}
\usepackage[noadjust]{cite}
\usepackage{epsfig}
\usepackage{color}
\newtheorem{theorem}{Theorem}
\newtheorem{lemma}[theorem]{Lemma}
\newtheorem{proposition}[theorem]{Proposition}

\newcommand\NN{{\mathbb N}}

\newcommand\ST{{\cal S}}
\renewcommand{\ge}{\geqslant}
\renewcommand{\le}{\leqslant}
\renewcommand{\geq}{\geqslant}
\renewcommand{\leq}{\leqslant}

\begin{document}
\title{Inducibility and universality for trees\thanks{The first and second authors were supported by the European Research Council (ERC) under the European Union's Horizon 2020 research and innovation programme (grant agreement No 648509). The second author was also supported by the MUNI Award in Science and Humanities (MUNI/I/1677/2018) of the Grant Agency of Masaryk university. The fourth author is supported by the Australian Research Council. }}
\author{Timothy F.~N. Chan\thanks{School of Mathematics, Monash University, Melbourne, Australia, and Mathematics Institute and DIMAP, University of Warwick, Coventry CV4 7AL, UK. E-mail: {\tt timothy.chan@monash.edu}.}
  \and Daniel Kr\'al'\thanks{Faculty of Informatics, Masaryk University, Botanick\'a 68A, 602 00 Brno, Czech Republic, and Mathematics Institute, DIMAP and Department of Computer Science, University of Warwick, Coventry CV4 7AL, UK. E-mail: {\tt dkral@fi.muni.cz}.}
	\and Bojan Mohar\thanks{Department of Mathematics, Simon Fraser University, Burnaby, BC, V5A 1S6, Canada. E-mail {\tt mohar@sfu.ca}. On leave from IMFM, Department of Mathematics, University of Ljubljana. B.M.~was supported in part by the NSERC Discovery Grant R611450 (Canada), and by the Research Project J1-2452 of ARRS (Slovenia).}
	\and David R. Wood\thanks{School of Mathematics, Monash University, Melbourne, Australia. E-mail: {\tt david.wood@monash.edu}. }
}
        
\date{}

\maketitle

\begin{abstract}
We answer three questions posed by Bubeck and Linial on the limit densities of subtrees in trees.
We prove there exist positive $\varepsilon_1$ and $\varepsilon_2$ such that
every tree that is neither a path nor a star has inducibility at most $1-\varepsilon_1$,
where the inducibility of a tree $T$ is defined as the maximum limit density of $T$, and
that there are infinitely many trees with inducibility at least $\varepsilon_2$.
Finally, we construct a universal sequence of trees; that is, a sequence in which the limit density of any tree is positive.
\end{abstract}

\section{Introduction}

Many results in extremal graph theory can be framed in terms of determining feasible combinations of subgraph densities, which are known as local profiles. For example, a recent breakthrough in extremal graph theory obtained by Razborov~\cite{Raz08}, Nikiforov~\cite{Nik11}, and Reiher~\cite{Rei16}, 
describes the possible densities of complete graphs in graphs with a given edge density;
also see~\cite{PikR17,LiuPS20} for results on the structure of extremal graphs.
Local profiles of other combinatorial structures such as tournaments have been studied by Linial and Morgenstern~\cite{LinM16}; 
also see~\cite{ChaGKN20,GrzKLV20}.
Another recent result on local profiles of graphs by Huang, Linial, Naves, Peled, and Sudakov~\cite{HuaLNPS14} determines the possible limit densities of $\overline{K_3}$ and $K_3$; also see~\cite{HuaLNPS16}. 
On the other hand, determining the minimum possible sum of densities of $\overline{K_4}$ and $K_4$
is one of the most intriguing problems in extremal graph theory,
and has remained open for more than five decades despite the work of
many researchers~\cite{Tho89,Tho97,FraR93,Fra02,JagST96,Gir79,Nie12,Spe11,Wol10}.
The challenging nature of the shape of graph profiles for some particular graphs
is in line with undecidability results on homomorphism inequalities by Hatami and Norine~\cite{HatN11}, and
results on the complex structure of graphs limits~\cite{GrzKL20,CooKM18}.
In this paper, we are concerned with local profiles of trees as studied by Bubeck and Linial~\cite{BubL16}; 
in particular, we answer three questions on the local profiles of trees posed by them.

In order to state our results precisely, we introduce the following definitions.
Let $T$ be a tree. We denote by $Z_k(T)$ the number of $k$-vertex subtrees in $T$.
An \emph{embedding} of a tree $S$ in $T$ is a subtree of $T$ isomorphic to $S$. Note that in our usage, an embedding can be associated with (possibly multiple) injective homomorphisms from $S$ to $T$, and all injective homomorphisms from $S$ to $T$ with the same image are associated with a single embedding.
The \emph{density} of a $k$-vertex tree $S$ in $T$, denoted by $d(S,T)$,
is the number of embeddings of $S$ in $T$ divided by $Z_k(T)$;
if the number of vertices of $T$ is less than $k$, we set $d(S,T)=0$.
The \emph{$k$-profile} of a tree $T$, denoted by $p^{(k)}(T)$,
is the vector whose entries are indexed by all non-isomorphic $k$-vertex trees, where 
the entry of $p^{(k)}(T)$ indexed by a tree $S$ is equal to $d(S,T)$.
Note that if the number of vertices of $T$ is less than $k$, then $p^{(k)}(T)$ is the zero vector, and
if the number of vertices of $T$ is at least $k$, then the entries of $p^{(k)}(T)$ sum to $1$.

We say that a sequence $(T_n)_{n\in\NN}$ of trees is \emph{convergent}
if the $k$-profiles $p^{(k)}(T_n)$ converge entrywise for every $k\in\NN$.
By Tychonoff's theorem, 
every sequence of trees has a convergent subsequence. The \emph{inducibility} of a tree $S$ is defined as the maximum limit density of $S$ in a convergent sequence of trees. In other words, the inducibility of $S$ is equal to
\[ \limsup_{n\to\infty} \max\{d(S,T) : \text{$T$ is an $n$-vertex tree} \} . \]
This concept was introduced for graphs by Pippenger and Golumbic~\cite{PipG75}; also see \cite{EL15,BS94,Yuster19,HHN14,Hirst14,Siran84,HT18,KNV19}. The definition for trees used here is by Bubeck and Linial~\cite{BubL16}, and it differs slightly from the definition used in~\cite{CzaSW17b,CzaDSW20,DosW18,DosW19}.

Clearly, paths have inducibility $1$ since every subtree of a path is a path. Similarly, stars have inducibility $1$ since every subtree of a star is a star.
Bubeck and Linial~\cite{BubL16} proved that paths and stars are the \textit{only} trees with inducibility $1$.
Motivated by this result, they asked~\cite[Problem 4]{BubL16} whether there are additional trees with inducibility arbitrarily close to $1$, or if not, whether there are infinitely many trees with inducibility bounded away from $0$ by a fixed constant:
\begin{itemize}
\item Does there exist $\varepsilon>0$ such that the inducibility of every tree that is neither a star nor a path is at most $1-\varepsilon$?
\item Does there exist $\varepsilon>0$ such that there are infinitely many trees with inducibility at least $\varepsilon$?
\end{itemize}
We answer both these questions affirmatively. The first question is answered in Theorem~\ref{thm:main} in Section~\ref{sec:main}. The proof relies on several preliminary results in Sections~\ref{sec:major}--\ref{sec:cater}. The second question is answered in Theorem~\ref{thm:sparkler} in Section~\ref{sec:sparkler}; this proof is self-contained. Both theorems give explicit values for $\varepsilon$, although we make no attempt to optimize these values. 

In the case of general graphs, it is well-known that the Erd\H os-R\'enyi random graph $G_{n,p}$
is a universal graph with high probability; that is, the limit density of every graph is positive in $G_{n,p}$.
Bubeck and Linial~\cite[Problem 5]{BubL16} asked whether there exist universal trees:
\begin{itemize}
\item Does there exist a convergent sequence $(T_n)_{n\in\NN}$ of trees in which the limit density $\lim\limits_{n\to \infty} d(S,T_n)$ of every tree $S$ is positive?
\end{itemize}
Our final result is an explicit construction of such a sequence of trees (Theorem~\ref{thm:universal} in Section~\ref{sec:universal}).

Regarding the state of the other problems appearing in \cite{BubL16}, Bubeck, Edwards, Mania and Supko~\cite{BubEMS16} and
Czabarka, Sz\'ekely and Wagner~\cite{CzaSW17a} independently resolved~\cite[Problem 3]{BubL16} by showing that
if the limit density of a $k$-vertex path $P_k$ in a (convergent) sequence of trees equals $0$,
then the limit density of the $k$-vertex star $S_k$ equals $1$. 
Further, results on $5$-profiles of trees can be found in~\cite{BubEMS16},
where additional questions raised in~\cite[Problems 1 and 7]{BubL16} have been answered.

\section{Preliminaries}

The number of vertices of a graph $G$ is denoted by $|G|$.
Given a vertex $v$ in a tree $T$, a \emph{branch} of $T$ rooted at $v$ is a subtree of $T$ formed by a component of the graph $T\setminus v$ together with its edge to $v$.
A branch is \emph{non-trivial} if it is not a single edge; in other words, it does not correspond to a leaf of $T$.
A non-trivial branch rooted at a vertex $v$ is a \emph{fork}
if it is isomorphic to a star (note that $v$ must be a leaf of this star).
The \emph{order} of a fork is its number of (non-root) leaves.
A branch is \emph{major} if it is a non-trivial branch that is not a fork.
A \emph{caterpillar} is a tree $T$ such that
every vertex of $T$ is the root of at most two non-trivial branches.
Finally, a vertex of a tree that is not a leaf is called \emph{internal}.
Observe that a tree is a caterpillar if and only if its internal vertices induce a path.

Czabarka, Sz\'{e}kely and Wagner~\cite[Theorem 1 and Lemma 4]{CzaSW17a} proved the following result about limit densities in trees of bounded radius.

\begin{proposition}[\cite{CzaSW17a}]
\label{prop:radius}
Let $(T_n)_{n\in\NN}$ be a convergent sequence of trees with $|T_n| \to \infty$.
If there exists an integer $K$ such that the radius of each $T_n$ is at most $K$, then
\[\lim_{n\to\infty}d(S_k,T_n)=1\]
for every $k\in\NN$, where $S_k$ is the $k$-vertex star.
\end{proposition}

As mentioned in the introduction, the result below is proved independently in~\cite[Theorem 2]{BubEMS16} and~\cite[Theorem 1]{CzaSW17a}.

\begin{proposition}[\cite{BubEMS16,CzaSW17a}]	
\label{prop:plimit}
Let $k\geq 4$ and let  $(T_n)_{n\in\NN}$ be a convergent sequence of trees with $|T_n| \to \infty$. If $\lim\limits_{n\to \infty} d(P_k,T_n)=0$,  then $\lim\limits_{n\to \infty} d(S_k,T_n)=1$.
\end{proposition}

A \emph{center} of a tree $T$ is a vertex $v$ such that each branch rooted at $v$ has at most $|T|/2$ edges.
Every tree $T$ has either one or two centers. Moreover, if $T$ has two centers, then $|T|$ is even, the two centers are adjacent, each center has a branch rooted at it with exactly $|T|/2$ edges, and the other center is its neighbor in this branch. A \emph{hub} of a tree $T$ is a vertex $v$ that is the only vertex on the path from $v$ to the nearest center of $T$ that is the root of at least three non-trivial branches. In particular, if a center of $T$ is the root of at least three non-trivial branches, then it is a hub.

\begin{proposition}
\label{prop:center}
Every tree $T$ that is not a caterpillar has at least one and at most two hubs.
\end{proposition}

\begin{proof}
Let $T'$ be the tree obtained from $T$ by removing all of its leaves. The degree of a vertex $v$ in $T'$ is equal to the number of non-trivial branches rooted at $v$ in $T$. Since $T$ is not a caterpillar, $T'$ is not a path. Therefore, $T'$ contains a vertex of degree at least $3$, so $T$ has at least one hub.

Let $W$ be the set of vertices of $T'$ with degree at least $3$.
Suppose that $T$ has a single center $v_C$.
If $v_C$ has degree at least three in $T'$,
then $v_C$ is the only hub of $T$.
Otherwise, the degree of $v_C$ in $T'$ is equal to $1$ or $2$ and
there exists at least one and at most two vertices $w \in W$ such that
there is no other vertex of $W$ on the unique path between $v_C$ and $w$.
These vertices $w$ are the hubs of $T$.

In the case that $T$ has two centers $v_C$ and $v'_C$,
which are necessarily adjacent,
then each center is a hub if its degree in $T'$ is at least three.
Otherwise,
there exists at most one vertex $w\in W$ such that
the unique path between $v_C$ and $w$ contains neither another vertex of $W$ nor $v'_C$. 
Similarly,
there exists at most one vertex $w\in W$ such that
the unique path between $v'_C$ and $w$ contains neither another vertex of $W$ nor $v_C$. 
Hence, $T$ has at most two hubs.
\end{proof}

Let $S_0$ and $S$ be embeddings of trees in a tree $T$ with $|S_0|=|S|=n$, and let $k$ be an integer less than $n$. (In fact, we only use $k\leq 3$). 
We say that $S$ can be obtained from $S_0$ by \emph{moving} $k$ edges if the intersection of $S_0$ and $S$ is a subtree of $T$ with $n-k$ vertices (see Figure~\ref{fig:move}). In this sense, $S$ is said to be obtained from $S_0$ by \emph{removing} the edges of $S_0$ that are not contained in $S$, and then \emph{adding} the edges of $S$ that are not contained in $S_0$. 

\begin{figure}
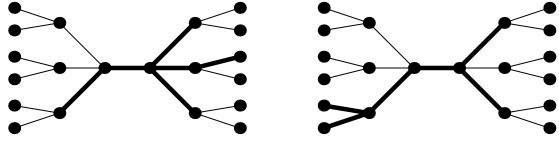

\begin{center}
\epsfbox{tprofile-7.mps}
\hskip 8mm
\epsfbox{tprofile-8.mps}
\end{center}
\caption{Two embeddings of $7$-vertex trees that can be obtained from each other by moving two edges. The edges of the embeddings are in bold.}
\label{fig:move}
\end{figure}

We next bound the number of vertices that can become a center of an embedding of a tree when at most three edges are moved.

\begin{proposition}
\label{prop:move}
Let $S_0$ be an embedding of a tree with at least $17$ vertices in another tree $T$.
There exists a set $X$ of at most $8$ vertices of $T$ such that
if three or fewer edges of $S_0$ are moved to produce an embedding $S$ of a tree in $T$,
then each center of $S$ is contained in $X$.
\end{proposition}

\begin{proof}
Let $n=|S_0|\ge 17$. 
Let $X$ be the set of vertices $v$ of $S_0$ such that each branch rooted at $v$ has at most $n/2+3$ edges.
We claim that $X$ has the property given in the statement of the lemma.
Indeed, if $S$ is an embedding obtained from $S_0$ by moving at most three edges and $w$ is a center of $S$,
then each branch of $S$ rooted at $w$ has at most $n/2$ edges and
so each branch of $S_0$ rooted at $w$ has at most $n/2+3$ edges.
Hence, $w$ is contained in $X$.

It remains to estimate $|X|$. 
We call a branch $B$ of $S_0$ \emph{significant} if $B$ is rooted at a center of $S_0$,
has at least $n/2-3$ edges, and
does not contain the other center (if another center exists). Every vertex $x\in X$ is either a center or is contained in a significant branch\textemdash otherwise, the branch rooted at $x$ containing the center(s) has at least $(n-1) - (n/2-4) + 1 = n/2+4$ edges.
Since significant branches are edge-disjoint and $3(n/2-3)=(n-1)+(n/2-8)>n-1$,
$S_0$ has at most two significant branches.
Note that each significant branch has at most $\lfloor n/2\rfloor$ edges
since it is rooted at a center of $S_0$, so the other branches rooted at the same center contain at least $\lceil n/2\rceil -1$ edges in total.

Therefore, if $n$ is odd, each significant branch has at most three vertices $w$ such that
the branch rooted at $w$ containing the center vertex has at most $\lceil n/2 \rceil +2 = \lfloor n/2 \rfloor +3$ edges.
If $n$ is even and $S_0$ has two centers, then the branches rooted at each center that contain the other center have exactly $n/2$ edges. So again, each significant branch has at most three vertices $w$ such that
the branch rooted at $w$ containing the center(s) has at most $n/2 + 3$ edges.
Lastly, if $n$ is even and $S_0$ has only one center,
then we use the fact that there is at most one significant branch with exactly $n/2$ edges.
This branch, if it exists, has at most four vertices $w$ such that
the branch rooted at $w$ containing the center has at most $n/2 + 3$ edges;
any other significant branch has at most three such vertices $w$.
In each case, $|X| \leq 8$. 
\end{proof}

Note that the bound on $|X|$ in Proposition~\ref{prop:move} is best possible since it is attained when $S_0$ is a path with an even number of vertices.

We finish this section by bounding the number
of vertices that can become a hub of an embedding of a tree when at most three edges are moved.

\begin{proposition}
\label{prop:move2}
Let $S$ be a non-caterpillar tree with at least $17$ vertices, and fix an embedding of a tree $S'$ with $|S'|=|S|$ in a tree $T$.
There exists a set $X$ of at most $144$ vertices of $T$ such that
if an embedding of $S$ in $T$ can be obtained by moving three or fewer edges of $S'$,
then each hub of the obtained embedding of $S$ is contained in $X$.
\end{proposition}

\begin{proof}
Let $X_0$ be the set of the vertices from Proposition~\ref{prop:move} applied with $S_0=S'$, and
let $D$ be the set of distances between the hubs of $S$ and the nearest center in $S$.
By Proposition~\ref{prop:center}, $S$ has at most two hubs, so $|D|\leq 2$.

For a vertex $z$ in the embedding of $S'$, define the \emph{resistance} of $z$ as
the number of edges not incident with $z$ that are contained in branches of $S'$ rooted at $z$
with the two largest branches excluded.
Informally speaking,
the resistance of $z$ is the number of edges that must be removed from $S'$ so that $z$ is no longer the root of three non-trivial branches, and therefore not a candidate hub.

Consider a vertex $x\in X_0$ that is
a center of an embedding of $S$ in $T$ obtained by moving at most three edges of $S'$.
Observe that in this embedding of $S$,
a vertex $v \neq x$ of $T$ can be a hub whose nearest center is $x$
only if the following holds:
\begin{itemize}
\item $v$ is an internal vertex of $S'$,
\item the distance $d$ between $v$ and $x$ belongs to the set $D$, and
\item the sum of the resistance of $x$ and the resistances of the internal vertices on the path between $v$ and $x$
      is at most $3$.
\end{itemize}
Let $X$ be the union of $X_0$ with the set of vertices $v$ that satisfy these three conditions for some $x\in X_0$.

For a vertex $x\in X_0$, let $Z_x$ be the union of $\{x \}$ with the set of internal vertices $z$ of $S'$ such that
the sum of the resistance of $x$ and
the resistances of the internal vertices on the path between $z$ and $x$
is at most $3$. 
Observe that if a vertex $z$ belongs to $Z_x$,
then all vertices on the path between $z$ and $x$ also belong to $Z_x$.
Define $S_Z$ to be the subtree of $S'$ induced by $Z_x$,
and note that the resistance of $z$ is an upper bound on the number of leaves of $S_Z$ lying in non-trivial branches rooted at $z$ with the two largest branches excluded.
Let $\delta$ be the number of branches of $S_Z$ rooted at $x$.
Since each of the $\delta$ branches of $S_Z$ rooted at $x$
has at most $4-\gamma \leq \min\{4,6-\delta\}$ leaves,
where $\gamma\ge\delta-2$ is the resistance of $x$,
the tree $S_Z$ has at most $9$ leaves.
This implies that the number of vertices of $Z_x$ lying at a distance contained in $D$ from $x$ is at most $18$.
Hence, the set $X$ contains at most $8\cdot 18=144$ vertices.
\end{proof}

\section{Inducibility of trees with three large branches}
\label{sec:major}

In this section, we present a part of the proof of Theorem~\ref{thm:main} for trees with three large branches rooted at a hub.
For a $k$-vertex tree $S$ and a host tree $T$,
one approach would be to construct a function $f_{S,T}$ that
maps each embedding of $S$ in $T$ to an embedding of a $k$-vertex subtree of $T$ non-isomorphic to $S$ such that at most $\alpha$ embeddings of $S$ are mapped to the same subtree of $T$, where $\alpha$ is  
a constant independent of $S$ and $T$. 
This would imply that the inducibility of $S$ is at most $\alpha / (\alpha + 1)$.
An explicit construction of such a function $f_{S,T}$ is technical,
so we prove its existence implicitly using a discharging argument.

\begin{theorem}
\label{thm:major}
Assume $S$ is a $k$-vertex tree $(k\ge 17)$ with a fixed hub $v_S$ that is either adjacent to at most one leaf, or is the root of at least three major branches and at most one fork.
If $T$ is a tree with radius at least $4k$,
then $d(S,T)\leq 1-10^{-7}$.
\end{theorem}

\begin{proof}
Let $v_T$ be a vertex of $T$ such that
there exist $(2k+1)$-vertex paths $P_1$ and $P_2$ starting at $v_T$ that
are disjoint except at $v_T$ itself; such a choice is possible because the radius of $T$ is at least $4k$.
For every vertex $v$ of $T$, fix a linear order $\preceq_v$ of the edges incident with $v$.

If there is at most one leaf adjacent to $v_S$,
then we say that every non-trivial branch rooted at $v_S$ is \emph{important};
otherwise, a branch rooted at $v_S$ is said to be \emph{important} only if it is major.
Note that there are at least three important branches regardless of which of the two cases
described in the statement of the theorem apply.
In this proof,
a \emph{stub} is an embedding of a $(k-3)$-vertex tree $S'$ in $T$ with a distinguished vertex $v'$ and three distinguished branches,
together with a correspondence between the distinguished branches and three (isomorphism classes of) branches of $S$ rooted at $v_S$ such that
it is possible to add a single leaf to each of the distinguished branches of $S'$ so that
there is an isomorphism from $S'$ to $S$ that maps $v'$ to $v_S$ and
the vertices of each of the distinguished branches of $S'$ to the vertices of the corresponding branch of $S$.
The three distinguished branches of the stub are referred to as \emph{grafts}.

We next introduce a canonical way of obtaining a stub from an embedding of $S$ in $T$.
For an embedding of $S$ in $T$,
let $v_{S\to T}$ be the vertex of $T$ corresponding to the hub $v_S$;
if there are two possible choices for $v_{S\to T}$, we choose an arbitrary one.
Consider the three important branches of the embedding rooted at $v_{S\to T}$
whose edges incident with $v_{S\to T}$ appear earliest in the linear order $\preceq_v$.
These three branches will be denoted by $S_A$, $S_B$, and $S_C$;
we will decide which branch is $S_A$, which is $S_B$, and which is $S_C$ later in the proof.
Consider the DFS traversal of the branches $S_A$, $S_B$, and $S_C$ from $v_{S\to T}$ such that
the edges at each vertex $v$ of the branches are visited in the order given by $\preceq_v$; that is, 
a part of the branch joined by an edge earlier in the order $\preceq_v$ is explored first.
We obtain the stub $S'$ by removing the leaf of the embedding that appears last in the DFS traversal in the branch $S_A$,
the leaf that appears last in the branch $S_B$, and
the leaf that appears last in the branch $S_C$.
The branches obtained from $S_A$, $S_B$, and $S_C$ are the grafts of the stub $S'$.

Let $v_A$, $v_B$, and $v_C$ be the vertices of the branches $S_A$, $S_B$, and $S_C$ adjacent
to the removed leaves of the embedding, and
let $d_A$, $d_B$, and $d_C$ be the degrees of $v_A$, $v_B$, and $v_C$ in the embedding of $S$, respectively.
Fix the indexing of the branches $S_A$, $S_B$, and $S_C$ so that $d_A\ge d_B\ge d_C\ge 2$.
The edges of $T$ incident with $v_A$, $v_B$, and $v_C$, respectively, that
appear in the orders $\preceq_{v_A}$, $\preceq_{v_B}$, and $\preceq_{v_C}$
after the edge of the embedding of $S$ that is visited second-last by the DFS traversal
are referred to as \emph{active},
with the possible exception of the edge towards the vertex $v_{S\to T}$; 
that is, the edge incident with $v_A$, $v_B$, or $v_C$ on the path to $v_{S\to T}$ is never active.
In particular, if $d_A=2$ then all edges incident with $v_A$ except the one towards the vertex $v_{S\to T}$ are active.
Observe that no active edges are contained in the stub $S'$ and
the only active edges contained in the embedding of $S$ are the three edges incident with the removed leaves.
Let $s_A$, $s_B$, and $s_C$ be the number of active edges incident with $v_A$, $v_B$, and $v_C$, respectively.

Observe that if a stub $S'$ is fixed, including the choice of the distinguished vertex and
grafts,
the vertices $v_A$, $v_B$ and $v_C$ are uniquely determined:
they are the last vertices in the DFS traversal uniquely determined by the orders $\preceq_v$
whose distance is $1$ less than the distance of the missing leaf.
Hence, the same stub $S'$ can be obtained from exactly $s_As_Bs_C$ embeddings of a tree $S$ in the tree $T$.

Define $Q$ to be the unique path in the tree $T$ between the vertices $v_{S\to T}$ and $v_T$
prolonged by $P_1$ if $P_1$ does not contain the edge from $v_T$ towards $v_{S\to T}$ and prolonged by $P_2$ otherwise.
Since $P_1$ and $P_2$ each have $2k+1$ vertices,
$Q$ has at least $k+1$ vertices not contained in the embedding of $S$.
We next construct a set $\ST$ of embeddings of several $k$-vertex trees non-isomorphic to $S$ in $T$ as follows.
\begin{itemize}
\item If $s_A\ge 3$ and $s_A\ge s_B$,
      then $\ST$ contains all embeddings obtained from $S'$
      by adding two active edges incident with $v_A$ and an active edge incident with $v_C$.
\item If $3\le s_A<s_B$ and $s_B\ge s_C$,
      then $\ST$ contains all embeddings obtained from $S'$
      by adding an active edge incident with $v_A$ and two active edges incident with $v_B$.
\item If $3\le s_A<s_B<s_C$ and $d_B\not=d_C+1$,
      then $\ST$ contains all embeddings obtained from $S'$
      by adding an active edge incident with $v_A$ and two active edges incident with $v_C$.
\item If $3\le s_A<s_B<s_C$ and $d_A-1\not=d_B=d_C+1$,
      then $\ST$ contains all embeddings obtained from $S'$
      by adding two active edges incident with $v_B$ and an active edge incident with $v_C$.
\item If $3\le s_A<s_B<s_C$ and $d_A-1=d_B=d_C+1$,
      then $\ST$ contains all embeddings obtained from $S'$
      by adding an active edge incident with $v_B$ and two active edges incident with $v_C$.
\item If $s_A<3$, $s_B\ge 3$ and $s_B\ge s_C$,
      then $\ST$ contains all embeddings obtained from $S'$
      by adding an active edge incident with $v_A$ and two active edges incident with $v_B$.
\item If $s_A<3\le s_C$, $s_B<s_C$ and $d_B\not=d_C+1$,
      then $\ST$ contains all embeddings obtained from $S'$
      by adding an active edge incident with $v_A$ and two active edges incident with $v_C$.
\item If $s_A<3\le s_C$, $s_B<s_C$ and $d_A\not=d_B=d_C+1$,
      then $\ST$ contains all embeddings obtained from $S'$
      by adding an active edge incident with $v_B$ and two active edges incident with $v_C$.
\item If $s_A<3\le s_C$, $s_B<s_C$ and $d_A=d_B=d_C+1$,
      then $\ST$ contains all embeddings obtained from $S'$
      by adding three active edges incident with $v_C$.
\item If $s_A$, $s_B$, and $s_C$ are all less than $3$,
      then $\ST$ contains the unique tree that is obtained from $S'$
      by adding the first three edges of $Q$ that are not already contained in $S'$.
\end{itemize}

The above ten cases cover all values of $s_A$, $s_B$, $s_C$, $d_A$, $d_B$, and $d_C$ satisfying $s_A, s_B, s_C \geq 1$ and $d_A \ge d_B \ge d_C \ge 2$.
In each case, the degree sequence of every tree in $\ST$ is different from the degree sequence of $S$.
For example, in the first case, trees in $\ST$ contain more vertices of degree $d_A+1$ than $S$, and
in the last case, the tree in $\ST$ has fewer leaves than $S$.
Therefore, none of the embeddings in $\ST$ is an embedding of a tree isomorphic to $S$.
In all but the last case, $|\ST| \geq s_As_Bs_C/12$.
For example, in the first case,
\[ |\ST| = \binom{s_A}{2} s_C = \frac{s_A(s_A-1) s_C}{2} \geq \frac{s_A^2 s_C}{3} \geq \frac{s_A s_B s_C}{3} \geq \frac{s_A s_B s_C}{12}. \]
In the last case, $|\ST|=1\ge s_As_Bs_C/8$.
This implies that $|\ST|$ is at least the number of embeddings of $S$ yielding the stub $S'$ divided by $12$.

Fix an embedding $S''$ of a $k$-vertex tree that is not isomorphic to $S$.
We now estimate the number of stubs $S'$ associated with an embedding of $S$ whose corresponding set $\ST$ contains $S''$.
We will create a stub $S'$ from $S''$ by following constructive steps that we next describe. The steps sometimes result in a tree that cannot be a stub of an embedding of the tree $S$, however,
any stub $S'$ associated with an embedding of $S$ such that the corresponding set $\ST$ contains $S''$
can be created by following the described steps.

The distinguished vertex of $S'$ must be a hub $v_{S \to T}$ of an embedding of $S$ that
can be transformed into $S''$ by moving at most three edges.
By Proposition~\ref{prop:move2}, there are at most $144$ choices for $v_{S\to T}$.
Once the vertex $v_{S\to T}$ is chosen,
it needs to be decided which three branches of $S''$ rooted at $v_{S\to T}$ correspond to grafts of the stub $S'$.
Suppose $S$ has at most one leaf adjacent to $v_S$,
which implies that every non-trivial branch rooted at $v_S$ is important.
Every graft of $S'$ corresponds to either a leaf (as
an important branch can be become trivial after the removal of a single edge\textemdash
note that at most three additional leaves can be created in this way), or
to one of the first four non-trivial branches rooted at $v_S$ in the order given by $\preceq_{v_S}$ (as a new non-trivial branch can be created by adding the first three edges of $Q$).
Hence, there are at most eight branches of $S''$ that could possibly be grafts in $S'$ when $v_{S\to T}$ is chosen.
Similarly, if $S$ has at most one fork rooted at $v_S$,
every graft of $S'$ corresponds to either a fork (as
an important branch can become a fork after the removal of a single edge\textemdash
again, at most three additional forks can be created in this way), or
to one of the first four major branches rooted at $v_S$ in the order given by $\preceq_{v_S}$.
Again, there are at most eight branches of $S''$ that could possibly be grafts in $S'$.

Next, fix a triple among the at most eight branches that could be the three grafts of $S'$.
Observe that the degree of $v_{S\to T}$ in $S''$ is the same as the degree of $v_S$ in $S$
unless a new branch at $v_{S\to T}$ was created by adding the first three edges on $Q$;
that is, these three edges form a branch rooted at $v_{S\to T}$ in $S''$.
In the latter case, remove the three edges of $Q$ that
have been added to get the same number of branches in the embedding as in $S'$.
The correspondence between the branches of the embedding and $S'$ different from the grafts
is given by their isomorphism to the branches of $S$ rooted at $v_S$.
Three branches of $S$ remain unmatched in this way and these can correspond in $3!=6$ ways to the grafts.
When the correspondence of these three branches and the grafts is fixed,
it is uniquely determined which edges of $S''$ need to be removed to get the stub $S'$.
For example,
if one of the branches of $S''$ has two additional edges but not two additional leaves compared to the corresponding branch of $S$, then the last of the ten cases applied (that is, three edges from $Q$ were added), and
we just remove the three edges of $Q$ to get $S'$.
Otherwise, 
the difference between the number of edges in the three branches of $S''$ chosen as grafts and the corresponding branches of $S$
determine the number of edges to be removed to get $S'$, and the correct edges to be removed are uniquely determined by
the linear orders $\preceq_v$.

We conclude that for every $k$-vertex tree $S''$, there are at most 
\[144 \cdot \tbinom{8}{3} \cdot 6\le 48\,384 \]
stubs $S'$ such that $S'$ is associated with an embedding of $S$ that the corresponding set $\ST$ contains $S''$;
the estimate follows from the fact that there are at most $144$ choices of $v_{S\to T}$,
each of which leads to at most $\binom{8}{3}$ choices of grafts and
at most six ways in which the grafts can correspond to the branches of $S$ rooted at $v_S$.

The bound on the density of $S$ in $T$ is obtained as follows.
Assign a charge of $48\,384\cdot 12=580\,608$ to each embedding of a $k$-vertex tree $S''$ in $T$ that
is not isomorphic to $S$.
Each such embedding sends $12$ units of charge to each of the at most $48\,384$ stubs $S'$ associated with an embedding of $S$ whose corresponding set $\ST$ contains $S''$.
In this way, every stub $S'$ receives at least $s_As_Bs_C$ units of charge,
where $s_A$, $s_B$ and $s_C$ are defined as above (note that
the quantities $s_A$, $s_B$ and $s_C$ are uniquely determined by the stub $S'$).
Finally, the stub $S'$ sends one unit of charge to each embedding of $S$ in $T$ whose associated stub is $S'$.
Since every embedding of $S$ in $T$ receives at least one unit of charge,
the density of $S$ in $T$ is at most $1-580\,609^{-1}\le 1-10^{-7}$.
\end{proof}

\section{Inducibility of trees with forks}
\label{sec:forks}

In this section, we analyze the inducibility of non-caterpillar trees that are not covered by Theorem \ref{thm:major}. A similar argument applies to a large class of caterpillars and so we formulate a single theorem to cover all cases.

\begin{theorem}
\label{thm:forks}
Let $S$ be a $k$-vertex tree $(k \geq 17)$ that
has a fixed vertex $v_S$ satisfying one of the following:
\begin{itemize}
\item $S$ is not a caterpillar, and $v_S$ is a hub of $S$ that is the root of at least one fork and is adjacent to at least two leaves,
\item $S$ is a caterpillar with at least four internal vertices, and $v_S$ is the root of a fork of order at least two and is adjacent to a leaf, or
\item $S$ is a caterpillar with exactly three internal vertices, and $v_S$ is the root of a fork and is adjacent to a leaf.
\end{itemize}
If $T$ is a tree with radius at least $4k$,
then $d(S,T) \leq 1-10^{-4}$.
\end{theorem}

\begin{proof}
The assumptions guarantee that $v_S$ is the root of at least two non-trivial branches, and that there is at most one vertex $v'_S\neq v_S$ of $S$ such that
$S$ has an automorphism mapping the vertex $v_S$ to $v'_S$.
Let $\ell$ be the maximum order of a fork rooted at $v_S$;
since $v_S$ is the root of a fork, $\ell>0$ is well-defined.

\textbf{Notation.}
Fix a tree $T$ with radius at least $4k$.
Let $v_T$ be a vertex of $T$ such that
there exist $(2k+1)$-vertex paths $P_1$ and $P_2$ starting at $v_T$ that
are disjoint except at $v_T$ itself; such a choice is possible because the radius of $T$ is at least $4k$.
We show that the density of $S$ in $T$ is at most $1-10^{-4}$ using a discharging argument that assigns
each embedding of a $k$-vertex tree non-isomorphic to $S$ a charge of $9\,999$ units and
redistributes this charge to embeddings of $S$ so that each one receives at least one unit of charge.

Consider an embedding of $S$ in $T$ and let $v_{S\to T}$ be the vertex of $T$ corresponding to $v_S$;
if there are two valid choices, choose $v_{S\to T}$ arbitrarily among them.
Let $R_0$ be the set of leaves of the embedding of $S$ adjacent to $v_{S\to T}$ and
let $R_i$ be the neighbors of $v_{S\to T}$ that are contained in a fork of order $i$ for $i\in\{1,\ldots,\ell\}$.
Note that $R_0\not=\emptyset$ and $R_\ell\not=\emptyset$.
In addition, observe that $\ell>1$ or $|R_0|>1$;
in the last case described in the statement of the lemma,
this is because $|S| \geq 17$. 
Set $R=R_0\cup\cdots\cup R_\ell$.
Let $\alpha$ be the number of edges of $T$ incident with $v_{S\to T}$ that are not contained in $S$, and
for a vertex $v\in R$, let $\beta_v$ be the number of edges incident with $v$ that are not contained in $S$.
Finally, define $Q$ to be the unique path in $T$ between $v_{S\to T}$ and $v_T$ prolonged by $P_1$ if $P_1$ does not contain the edge from $v_T$ towards $v_{S\to T}$ and prolonged by $P_2$ otherwise.
Since $P_1$ and $P_2$ each have $2k+1$ vertices,
$Q$ has at least $k+1$ vertices not contained in the embedding of $S$.

\textbf{Definition of correspondence.}
We next define sets $\ST_A$, $\ST_B$ and $\ST_C$ of embeddings of trees non-isomorphic to $S$, and
in some cases, we also define a set $\ST_D$.
Each of the embeddings contained in $\ST_A$, $\ST_B$ and $\ST_C$
can be obtained from the embedding of $S$ by moving an edge, and
some of these sets can be empty.

Let $\ST_A$ be the set of embeddings obtained by removing a leaf adjacent to $v_{S\to T}$ and
adding a leaf to a fork of order $\ell$ rooted at $v_{S\to T}$.
Note that the number of leaves of $S$ that are adjacent to a vertex that is the root of exactly one non-trivial branch (that is, the number of leaves contained in a fork) 
is one fewer than the number of such leaves in the obtained embedding.
Hence, the trees in $\ST_A$ are not isomorphic to $S$.
Observe that
\[|\ST_A|=|R_0|\cdot\sum_{v\in R_\ell}\beta_v,\]
and let
\[\varepsilon_A=\frac{|\ST_A|}{(\ell+1)(\alpha+1)}=\frac{|R_0|\cdot\sum_{v\in R_\ell}\beta_v}{(\ell+1)(\alpha+1)}.\]

Let $\ST_B$ be the set of embeddings obtained by removing a leaf adjacent to $v_{S\to T}$ and adding a leaf to a fork of order $\ell-1$ if $\ell\ge 2$ or
adding a leaf to another leaf adjacent to $v_{S\to T}$ if $\ell=1$.
Since the number of leaves of $S$ that are adjacent to a vertex that is the root of exactly one non-trivial branch
is one fewer than the number of such leaves in the obtained embedding,
the trees in $\ST_B$ are not isomorphic to $S$.
Observe that if $\ell\not=1$, then
\[|\ST_B|=|R_0|\cdot\sum_{w\in R_{\ell-1}}\beta_w,\]
and if $\ell=1$, then
\[|\ST_B|=(|R_0|-1)\cdot\sum_{w\in R_{\ell-1}}\beta_w.\]
Finally, let
\[\varepsilon_B=\frac{|\ST_B|}{(|R_\ell|+1)\ell(\alpha+1)}\ge\frac{|R_0|\cdot\sum_{w\in R_{\ell-1}}\beta_w}{2(|R_\ell|+1)\ell(\alpha+1)};\]
the inequality holds since $\ell>1$ or $|R_0|>1$.

Next, let $\ST_C$ be the set of embeddings obtained by removing a leaf of a fork of order $\ell$ rooted at $v_{S\to T}$ and
adding a leaf adjacent to $v_{S\to T}$.
Unless $\ell=1$,
the number of leaves of $S$ that are adjacent to a vertex that is the root of exactly one non-trivial branch
is one more than the number of such leaves in the obtained embedding.
If $\ell=1$, then the number of leaves of $S$ is one less than the number of leaves in the obtained embedding.
In both cases, the trees contained in $\ST_C$ are non-isomorphic to $S$.
Observe that
\[|\ST_C|=|R_\ell|\cdot\ell\cdot\alpha,\]
and let
\[\varepsilon_C=\sum_{v\in R_\ell}\frac{\ell\cdot\alpha}{(|R_0|+1)\left(\beta_v+1+\sum_{w\in R_{\ell-1}}\beta_w\right)}.\]

If $\ell\ge 2$, then we also define $\ST_D$ to be the set of embeddings obtained by removing a leaf adjacent to $v_{S\to T}$ and
a leaf of a fork of order $\ell$ rooted at $v_{S\to T}$, and
adding the first two edges of $Q$ not contained in the embedded tree.
Since the number of leaves of $S$ is at least one more than the number of leaves in the obtained embedding,
the obtained embedding is not an embedding of $S$.
Unlike in the previous three cases,
the embeddings obtained in this way need not all be embeddings of the same tree
since one of the removed edges can be contained in $Q$ and then added back.
Observe that
\[|\ST_D|=|R_0|\cdot|R_\ell|\cdot\ell,\]
and let
\[\varepsilon_D=\sum_{v\in R_\ell}\frac{|R_0|\cdot\ell}{(\alpha+1)\left(\beta_v+1+\sum_{w\in R_{\ell-1}}\beta_w\right)}.\]

\textbf{Discharging argument.}
Given an embedding $S'$ of a $k$-vertex tree in $T$,
the number of choices of a vertex $v_{S\to T}$ in $S'$ such that
$S'$ is contained in one of the sets $\ST_A$, $\ST_B$, $\ST_C$, and $\ST_D$
for an embedding of $S$ with the vertex $v_S$ mapped to $v_{S\to T}$
is at most $144$ by Proposition~\ref{prop:move2} if $S$ is not a caterpillar.
If $S$ is a caterpillar, then the number of choices of a vertex $v_{S\to T}$ in $S'$ such that
$S'$ is contained in $\ST_A$, $\ST_B$, or $\ST_C$
for an embedding of $S$ with the vertex $v_S$ mapped to $v_{S\to T}$ is at most 4:
if $S'$ is a caterpillar then $v_{S\to T}$
must be its first, second, second-last, or last internal vertex, and
if $S'$ is not a caterpillar,
then it can only be contained in $\ST_B$ and $v_{S\to T}$ is its unique vertex with two forks.
Furthermore, 
the number of choices of a vertex $v_{S\to T}$ in $S'$ such that
$S'$ is contained in $\ST_D$ for an embedding of $S$ with the vertex $v_S$ mapped to $v_{S\to T}$ is at most $4$:
there are at most two choices of edges that could have been added as part of $Q$ and,
once these edges are chosen and removed,
$v_{S\to T}$ is either its second or second-last internal vertex (assuming the tree is a caterpillar).
We conclude that
if $S$ is a caterpillar, then the number of choices of a vertex $v_{S\to T}$ in $S'$ such that
$S'$ is contained in one of the sets $\ST_A$, $\ST_B$, $\ST_C$ and $\ST_D$
for an embedding of $S$ with the vertex $v_S$ mapped to $v_{S\to T}$
is at most $8$.

For each choice of $v_{S\to T}$,
the embedding $S'$ distributes $10$ units of its charge equally to the embeddings of $S$ with $v_{S\to T}$ such that
$S'$ is in the corresponding set $\ST_A$, $10$ units of its charge equally to the embeddings such that $S'$ is in the corresponding set $\ST_B$, $10$ units of its charge equally to the embeddings such that $S'$ is in the corresponding set $\ST_C$, and,
if $\ell\ge 2$,
an additional $10$ units of its charge equally to the embeddings such that $S'$ is in the corresponding set $\ST_D$.
In this way, the embedding $S'$ distributes at most $8 \cdot 40 = 320$ charge if $S$ is a caterpillar, and at most $144 \cdot 40 = 5\,760$ units of charge if it is not.
We remark that there will be additional charge distributed by $S'$ by rules described later in the proof.
Each embedding of $S$ receives at least $10(\varepsilon_A+\varepsilon_B+\varepsilon_C)$ units of charge and,
if $\ell\ge 2$,
at least $10(\varepsilon_A+\varepsilon_B+\varepsilon_C+\varepsilon_D)$ units of charge.
In particular, the considered embedding receives at least one unit of charge
unless $\varepsilon_A$, $\varepsilon_B$, $\varepsilon_C$, and $\varepsilon_D$ are all less than $1/10$.

We next show that one of $\varepsilon_A$, $\varepsilon_B$ and $\varepsilon_C$ is at least $1/10$
unless $\alpha=0$ or $\sum_{w\in R_{\ell-1}}\beta_w=\sum_{v\in R_\ell}\beta_v=0$.
Suppose that $\alpha\not=0$.
Let $B$ be the maximum value of $\beta_v$ for $v\in R_\ell$.
If $B>\sum_{w\in R_{\ell-1}}\beta_w$, then 
\begin{align*}
\varepsilon_A\varepsilon_C & =\frac{|R_0|\cdot\sum_{v\in R_\ell}\beta_v}{(\ell+1)(\alpha+1)}\cdot
\sum_{v\in R_\ell}\frac{\ell\cdot\alpha}{(|R_0|+1)\left(\beta_v+1+\sum_{w\in R_{\ell-1}}\beta_w\right)}\\
& \ge \frac{|R_0|\cdot\sum_{v\in R_\ell}\beta_v}{4\ell\alpha}\cdot
\frac{|R_\ell|\cdot\ell\cdot\alpha}{2(|R_0|+1)B}\ge \frac{|R_0|}{8(|R_0|+1)}\ge\frac{1}{16}.
\end{align*}
Hence, $\varepsilon_A$ or $\varepsilon_C$ is at least $1/10$.
If $B\le\sum_{w\in R_{\ell-1}}\beta_w$ and $\sum_{w\in R_{\ell-1}}\beta_w\not=0$, then 
\begin{align*}
\varepsilon_B\varepsilon_C & \geq \frac{|R_0|\cdot\sum_{w\in R_{\ell-1}}\beta_w}{2(|R_\ell|+1)\ell(\alpha+1)}\cdot
                               \sum_{v\in R_\ell}\frac{\ell\cdot\alpha}{(|R_0|+1)\left(\beta_v+1+\sum_{w\in R_{\ell-1}}\beta_w\right)}\\
                           & \ge \frac{|R_0|\cdot\sum_{w\in R_{\ell-1}}\beta_w}{8|R_\ell|\ell\alpha}\cdot
			         \sum_{v\in R_\ell}\frac{\ell\cdot\alpha}{3(|R_0|+1)\sum_{w\in R_{\ell-1}}\beta_w}\\
                           & = \frac{|R_0|}{24(|R_0|+1)}\ge\frac{1}{48}.
\end{align*}
Hence, $\varepsilon_B$ or $\varepsilon_C$ is at least $1/10$.
We conclude that if $\alpha\not=0$,
then one of $\varepsilon_A$, $\varepsilon_B$ and $\varepsilon_C$ is at least $1/10$
unless $\sum_{w\in R_{\ell-1}}\beta_w=0$ and $\sum_{v\in R_\ell}\beta_v=0$.
So, we need to analyze the cases
when $\alpha=0$ or when $\sum_{w\in R_{\ell-1}}\beta_w=\sum_{v\in R_\ell}\beta_v=0$.

\textbf{Analysis of non-caterpillars.}
Suppose that $S$ is not a caterpillar and $\alpha=0$.
Let $S'$ be obtained from the embedding of $S$ by removing any two leaves adjacent to $v_{S\to T}$ and
adding the first two edges on $Q$ not contained in the embedding.
Since $S'$ has at least one less leaf than $S$, it is not isomorphic to $S$.
The embedding $S'$ sends one unit of charge to the considered embedding of $S$.
Note that the embedding $S'$ sends by this rule at most $2\cdot144$ units of charge
in addition to the charge sent earlier: when $S'$ is fixed, there are at most two choices of edges that could have been added as part of $Q$, and at most $144$ choices of $v_{S\to T}$ by Proposition~\ref{prop:move2}.
The leaves adjacent to $v_{S\to T}$ that were removed are uniquely determined since $\alpha=0$.

Suppose that $S$ is not a caterpillar, $\alpha>0$ and $\sum_{w\in R_{\ell-1}}\beta_w=\sum_{v\in R_\ell}\beta_v=0$.
If $\ell\ge 2$, then
\[\varepsilon_C=\frac{|R_\ell|\ell\alpha}{|R_0|+1}\qquad\mbox{and}\qquad
\varepsilon_D=\frac{|R_\ell|\ell|R_0|}{\alpha+1}.\]
It follows that $\varepsilon_C\varepsilon_D\ge 1/4$; that is, $\varepsilon_C$ or $\varepsilon_D$ is at least $1/10$.
If $\ell=1$ and $|R_1|\ge 2$, then 
let $S'$ be obtained from the embedding of $S$ by removing a leaf from two forks rooted at $v_{S\to T}$ and
adding the first two edges of $Q$.
Since the embedding $S'$ has fewer leaves contained in forks, $S'$ is not isomorphic to $S$.
The embedding $S'$ sends one unit of charge to the considered embedding of $S$.
Each embedding $S'$ sends in this way at most $2\cdot144$ units of charge in addition to the charge sent earlier: when $S'$ is fixed, there are at most two choices of edges that could have been added as part of $Q$, and at most $144$ choices of $v_{S\to T}$ by Proposition~\ref{prop:move2}.
The leaves adjacent to $v_{S\to T}$ to be changed to a fork are uniquely determined and so are the edges to be added
since $\sum_{w\in R_{\ell-1}}\beta_w=\sum_{v\in R_\ell}\beta_v=0$.

If $\ell=|R_1|=1$, then let $e$ be the edge incident with the leaf of the fork rooted at $v_{S\to T}$.
If $v_{S\to T}$ has a neighbor $w$ in $T$ that is not contained in $S$ and that has degree at least two in $T$, then 
remove the edge $e$ and add the edge $v_{S\to T}w$ to obtain an embedding $S'$.
The embedding $S'$ has more leaves than $S$ and so is not isomorphic to $S$.
The embedding $S'$ sends one unit of charge to the considered embedding of $S$.
Each embedding $S'$ sends in this way at most $2\cdot 144$ units of charge in addition to the charge sent earlier as
there are at most two leaves adjacent to $v_{S\to T}$ that can be changed to a fork with the unique edges to be added (as
$\sum_{w\in R_{\ell-1}}\beta_w=\sum_{v\in R_\ell}\beta_v=0$).
Hence, we can assume that all neighbors of $v_{S\to T}$ in $T$ are leaves except its neighbors that are
contained in the non-trivial branches of $S$.

If $e$ is not contained in $Q$,
then let $S'$ be the embedding obtained from $S$ by removing $e$ and adding the first edge of $Q$ not contained in $S$, and
let $S''$ be the embedding obtained from $S$ by removing the fork containing $e$ and adding the first two edges of $Q$ not contained in $S$.
Observe that $S'$ or $S''$ is not isomorphic to $S$ since at least one of them has a different number of leaves from $S$.
The embedding that is not isomorphic sends one unit of charge to $S$ and
each embedding sends at most $4\cdot 144$ units of charge in this way (it can appear in the role of $S'$ and $S''$, there are at most two choices of edges that could have been added as part of $Q$, and there are at most $144$ choices of $v_{S\to T}$, each determining the embedding $S$ uniquely).

If $e$ is contained in $Q$, then
let $\ST_1$ be the set of embeddings obtained by removing a leaf adjacent to $v_{S\to T}$ and
adding the edge of $Q$ following the edge $e$;
note that $|\ST_1|=|R_0|$.
Let $\ST_2$ be the set of embeddings obtained by removing the edge $e$ and
adding an edge incident with $v_{S\to T}$ not contained in $S$;
note that $|\ST_2|=\alpha$.
Since the embeddings in $\ST_1$ and $\ST_2$ have different numbers of leaves than $S$,
they are not isomorphic to $S$.
Each embedding in $\ST_1$ distributes one unit of charge equally among all $\alpha+1$ embeddings of $S$ that can be obtained in this way, and
each embedding in $\ST_2$ distributes one unit of charge equally among all $|R_0|+1$ embeddings of $S$ that can be obtained in this way (note that $v_{S\to T}$ is uniquely determined as
the vertex in the embedding with degree greater than 2 that is closest to the added edges of $Q$, and
the fork of an embedding of $S$ is created only by adding an edge to a leaf adjacent to $v_{S\to T}$ whose degree in $T$ is $2$).
We conclude that the embedding of $S$ receives at least
\[\frac{|R_0|}{\alpha+1}+\frac{\alpha}{|R_0|+1}\ge \frac{1}{2}\left(\frac{|R_0|}{\alpha}+\frac{\alpha}{|R_0|}\right)\ge 1\]
units of charge.

\textbf{Analysis of caterpillars.}
We next analyze the case when $S$ is a caterpillar.
If $\alpha=0$ and $|R_0|\ge 2$, then 
let $S'$ be obtained from the embedding of $S$ by removing any two leaves adjacent to $v_{S\to T}$ and
adding the first two edges on $Q$ not contained in the embedding.
Since $S'$ has fewer leaves than $S$, $S'$ is not isomorphic to $S$.
The embedding $S'$ sends one unit of charge to the considered embedding of $S$.
Note that in this way the embedding $S'$
sends at most $2\cdot 2$ units of charge in addition to the charge sent earlier:
there are at most two choices of edges that could have been added as part of $Q$ and,
once these edges are chosen and removed,
the vertex $v_{S\to T}$ is either the second or second-last internal vertex of the resulting caterpillar.

If $\alpha=0$, $|R_0|=1$ and $\ell\ge 2$, then 
we derive along the lines used in the general case that
$\varepsilon_A\varepsilon_D\ge\frac{|R_0|^2}{4(\alpha+1)^2}=\frac{1}{4}$ or
$\varepsilon_B\varepsilon_D\ge\frac{|R_0|^2}{12(\alpha+1)^2}=\frac{1}{12}$
unless $\sum_{w\in R_{\ell-1}}\beta_w=\sum_{v\in R_\ell}\beta_v=0$.
However, if $\sum_{w\in R_{\ell-1}}\beta_w=\sum_{v\in R_\ell}\beta_v=0$,
then $\varepsilon_D=\frac{|R_0|\cdot|R_\ell|\cdot\ell}{\alpha+1}\ge 2$.

If $\alpha=0$ and $|R_0|=\ell=1$,
then $\varepsilon_A\ge 1/2$ unless $\sum_{v\in R_\ell}\beta_v=0$.
If $\alpha=0$, $|R_0|=\ell=1$, and $\sum_{v\in R_\ell}\beta_v=0$,
and more generally whenever $\ell=1$ and $\sum_{v\in R_\ell}\beta_v=0$,
then we are in the third case from the statement of the lemma, and $|R_\ell|=2$.
In other words, $S$ is a star with two different edges subdivided.
Consider $S'$ obtained from the embedding of $S$ by removing the edges that are incident with the leaves of the two forks of $S$ and
adding the first two edges on $Q$ not contained in the embedding.
Observe that $S'$ is not isomorphic to $S$.
The embedding $S'$ sends one unit of charge to the considered embedding of $S$.
Note that in this way the embedding $S'$
sends at most one unit of charge in addition to the charge sent earlier;
the edges of $S'$ that were added as a part of $Q$ are the unique edges whose removal creates a star,
the vertex $v_{S\to T}$ is the internal vertex of this star,
and the remaining two edges of the embedding of $S$ are uniquely determined since $\sum_{v\in R_\ell}\beta_v=0$.

The final case to consider is when $\alpha>0$, $\sum_{w\in R_{\ell-1}}\beta_w=\sum_{v\in R_\ell}\beta_v=0$, and $\ell\ge 2$ (note that
the case $\ell=1$ is covered in the previous paragraph).
As in the non-caterpillar case, it follows that $\varepsilon_C\varepsilon_D\ge \ell^2 |R_\ell|^2/4 \ge 1/4$; that is, $\varepsilon_C$ or $\varepsilon_D$ is at least $1/10$.

\textbf{Conclusion.}
According to the rules set above, each embedding of a tree non-isomorphic to $S$ distributes
at most $320+4+1=325$ units of charge if $S$ is a caterpillar, and at most $5\,760+10\cdot 144+2=7\,202$ units of charge if it is not.
Thus, the density of $S$ in $T$ is at most $1-7203^{-1} \leq 1-10^{-4}$.
\end{proof}

\section{Inducibility of caterpillars}
\label{sec:cater}

In this section, we complete the analysis of the inducibility of caterpillars.
We start with caterpillars whose second or second-last internal vertex
is the root of a fork of order 1 and is adjacent to a leaf.

\begin{lemma}
\label{lem:cater2}
Let $S$ be a non-path caterpillar with $k\ge 10$ vertices that
has at least four internal vertices and
has a fixed vertex $v_S$ that is the root of a fork of order $1$ and is adjacent to a leaf.
If $T$ is a tree with radius at least $4k$,
then $d(S,T) \leq 1-10^{-3}$.
\end{lemma}

\begin{proof}
Let $\ell > 0$ be the number of leaves adjacent to $v_S$.
Since $S$ is a caterpillar with at least four internal vertices, $v_S$ is the root of exactly one fork;
the order of this fork is $1$ by the assumption of the lemma.

Let $v_T$ be a vertex of $T$ such that
there exist $(2k+1)$-vertex paths $P_1$ and $P_2$ starting at $v_T$ that
are disjoint except at $v_T$ itself; such a choice
is possible because the radius of $T$ is at least $4k$.

\begin{figure}
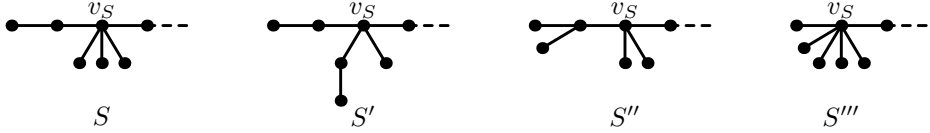

\begin{center}
\epsfbox{tprofile-3.mps}
\hskip 8mm
\epsfbox{tprofile-4.mps}
\hskip 8mm
\epsfbox{tprofile-5.mps}
\hskip 8mm
\epsfbox{tprofile-6.mps}
\end{center}
\caption{The trees $S$, $S'$, $S''$ and $S'''$ used in the proof of Lemma~\ref{lem:cater2} when $\ell=3$.}
\label{fig:cater2}
\end{figure}

We next define trees $S_0$, $S'$, $S''$ and $S'''$ (see Figure~\ref{fig:cater2}).
Let $S_0$ be the tree obtained from $S$ by removing the fork and all leaves adjacent to $v_S$.
If $\ell\ge 2$, then let $S'$ be the tree obtained from $S$ by removing a leaf adjacent to $v_S$ and
turning another leaf adjacent to $v_S$ into a fork of order $1$; if $\ell=1$, then $S'$ is not defined.
Let $S''$ be the tree obtained from $S$ by removing a leaf adjacent to $v_S$ and adding a leaf to the fork rooted at $v_S$. 
Finally, let $S'''$ be the tree obtained from $S$ be removing the leaf of the fork rooted at $v_S$ and adding a leaf adjacent to $v_S$.
The trees $S$, $S'$ (if defined), $S''$, and $S'''$ are mutually non-isomorphic, since 
$S'$ is the only one that is not a caterpillar, 
$S'''$ has fewer internal vertices than $S$ and $S''$, and
the numbers of leaves adjacent to the first and last internal vertices of $S$ and $S''$ differ.

In this proof, a \emph{stub} is an embedding of $S_0$ with a distinguished vertex $v_{S\to T}$ such that the embedding of $S_0$ can be extended to an embedding of $S$ with $v_{S\to T}$ corresponding to $v_S$.
Let $Q$ be the unique path in $T$ between the vertices $v_{S\to T}$ and $v_T$
prolonged by $P_1$ if $P_1$ does not contain the edge from $v_T$ towards $v_{S\to T}$ and prolonged by $P_2$ otherwise.
Let $D \ge \ell +1$ be the degree of $v_{S\to T}$ in $T$ minus $1$ and
let $d_1,\ldots,d_D$ be the degrees of its neighbors not contained in the embedding of $S_0$ minus $1$.
The numbers of ways that the embedding of $S_0$ can be extended (with $v_{S\to T}$ corresponding to $v_S$) to an embedding of $S$, $S'$, $S''$, and $S'''$ are
\[\sum_{1\le i\le D} d_i\binom{D-1}{\ell},
  \sum_{1\le i<j\le D} d_id_j\binom{D-2}{\ell-2},
  \sum_{1\le i\le D} \binom{d_i}{2}\binom{D-1}{\ell-1},\mbox{ and }
  \binom{D}{\ell+2},\]
respectively.
Let $N$, $N'$, $N''$ and $N'''$ be these numbers. We claim that $N\le 54(N'+N''+N'''+N_Q)$,
where $N_Q$ is the number of extensions of $S_0$ to an embedding of a tree $S_Q$ that is defined later.
Note that $N_Q>0$ only if $\sum_{1\le i\le D} d_i=1$ or $\ell=d_1=d_2=1$; 
we set $N_Q=0$ otherwise.

The following three paragraphs concern the case when $\sum_{1\le i\le D} d_i=1$;
note that $N'=N''=0$ in this case.
Fix an embedding of $S$ and denote the internal vertices of the embedding by $v_1, \dots, v_m$ so that $v_2 = v_S$.
Consider the first edge of $Q$ not contained in this embedding.
If this edge is incident with $v_m$,
then let $S_Q$ be the embedding obtained from $S$ by removing one of the leaves adjacent to $v_{S\to T}$ and
adding the first unused edge of $Q$.
The sum of the degrees of the first and last internal vertices of $S_Q$ is greater than that of $S$,
so $S_Q$ is not isomorphic to $S$.
Furthermore, the number $N_Q$ of ways that the embedding of $S_0$
can be extended to an embedding of $S_Q$ (with the added edge of $Q$ fixed) is $\binom{D-1}{\ell-1}$.

Now suppose that the first edge of $Q$ not contained in the embedding of $S$ is incident with a vertex other than $v_m$;
note that this edge is not incident with $v_2$ because $\sum_{1\le i\le D} d_i=1$.
Let $S_Q$ be the embedding obtained from $S$ by removing the leaf of the fork rooted at $v_{S\to T}$ and one of the leaves adjacent to $v_{S\to T}$, and
then adding the first two unused edges of $Q$.
The resulting embedding is not isomorphic to $S$:
if the first unused edge is incident with one of $v_3,\dots,v_{m-1}$ or their adjacent leaves, then $S_Q$ is not a caterpillar, and
if the first unused edge is incident with a leaf of $v_1$ or $v_m$, then $S_Q$ has more internal vertices than $S$.
Again, the number $N_Q$ of ways that the embedding of $S_0$
can be extended to an embedding of $S_Q$ (with the edge $v_1 v_2$ and the added edges of $Q$ fixed) is $\binom{D-1}{\ell-1}$.

It follows that
\begin{align*}
N & =\binom{D-1}{\ell} =\binom{D-1}{\ell-1} \frac{D-\ell}{\ell}, \\
N''' & =\binom{D}{\ell+2} =\binom{D-1}{\ell-1} \frac{D(D-\ell)(D-\ell-1)}{(\ell+2)(\ell+1)\ell}, \text{ and}\\
N_Q &= \binom{D-1}{\ell-1}.
\end{align*}
If $D\le 2\ell$, then $N\le N_Q$.
If $D\ge 2\ell+1$, then $D - \ell - 1 \geq \ell \geq \frac{1}{2}(\ell + 1)$ and $D \geq \ell+ 2$, so $N\le 2N'''$.
Therefore, regardless of the relationship between $D$ and $\ell$, we have $N \leq 2 ( N''' + N_Q)$.

In the rest of the proof, we analyze the case when $\sum_{1\le i\le D} d_i\ge 2$.
If $D=\ell+1$, then the numbers of ways that the embedding of $S_0$
can be extended to an embedding of $S$, $S'$ and $S''$ (note that $N'''=0$) are
\[N=\sum_{1\le i\le D} d_i,\quad
  N'=\sum_{1\le i<j\le D}(\ell-1) d_id_j,\mbox{ and }
  N''=\sum_{1\le i\le D}\ell \binom{d_i}{2},\]
respectively.
If $\ell\ge 2$, then $N\le 2N'$ unless only one of the $d_i$ is non-zero, in which case $N\le N''$.
If $\ell=1$, then $N=d_1+d_2$, $N'=0$ and $N''=\binom{d_1}{2}+\binom{d_2}{2}$,
in which case $N\le 2N''+2\le 4N''$ unless $d_1=d_2=1$.
Finally, if $\ell=d_1=d_2=1$, then consider the embedding $S_Q$ defined in the same way as in the case $\sum_{1\le i\le D} d_i=1$
unless the first edge of $Q$ not contained in the embedding of $S$ is incident with the leaf adjacent to $v_{S\to T}$;
note that the embedding $S_Q$ is well-defined as 
the first edge of $Q$ not contained in the embedding of $S$ cannot be incident with $v_2$ as $D=\ell+1$.
If the first edge of $Q$ not contained in the embedding of $S$ is incident with the leaf adjacent to $v_{S\to T}$, then remove the fork rooted at $v_{S\to T}$ and add the first two edges of $Q$ not contained in $S$;
the resulting embedding $S_Q$ is a caterpillar with diameter greater than that of $S$ and so is non-isomorphic to $S$
(note that the embedding $S_Q$ is the same for both embeddings of $S$ that can be obtained from the same stub).
In all cases describe above,
the embedding $S_Q$ is uniquely determined by the embedding of the stub, so $N_Q=1$, 
which implies that $N=2\le 2N_Q$.

Next assume that $D\ge\ell+2$.
If $\ell\ge 2$, then 
\begin{align*}
N & \le 2\binom{D-2}{\ell-2}\sum_{1\le i\le D} d_i\frac{D(D-\ell)}{\ell^2}, \\
N' & \ge \binom{D-2}{\ell-2}\sum_{1\le i<j\le D} d_id_j, \\
N'' & \ge \binom{D-2}{\ell-2}\sum_{1\le i\le D} \binom{d_i}{2},\text{ and} \\
N''' & \ge \frac{1}{8}\binom{D-2}{\ell-2}\frac{D^2(D-\ell)^2}{\ell^4}
\end{align*}
Hence, 
\[N'+N''\ge \frac{1}{4}\binom{D-2}{\ell-2}\left(\sum_{1\le i\le D}d_i\right)^2,\]
which implies that $N \le 8(N'+N''+N''')$ by the AM-GM inequality.

If $\ell=1$, then 
\begin{align*}
N & \le \sum_{1\le i\le D}Dd_i
    \le D^2+\sum_{\substack{1\le i\le D \\ d_i \geq 2}}D(d_i-1)
    \le D^2+D^3+\sum_{\substack{1\le i\le D \\ d_i \geq 2}} (d_i-1)^2, \\
N'' & =\sum_{1\le i\le D} \binom{d_i}{2}\ge \sum_{\substack{1\le i\le D \\ d_i \geq 2}} \frac{(d_i-1)^2}{2},\mbox{ and } \\
N''' & =\binom{D}{3}\ge \frac{D^3}{27}.
\end{align*}
It follows that
\begin{align*}
N \leq D^2+D^3+\sum_{\substack{1\le i\le D \\ d_i \geq 2}}(d_i-1)^2 \leq 54(N'' + N''').
\end{align*}

In all cases, we have proved that $N\le 54(N'+N''+N'''+N_Q)$.
Each embedding of a $k$-vertex tree non-isomorphic to $S$ sends $72$ charge to each stub that it extends.
In this way, each stub receives at least $54 (N' + N'' + N''' + N_Q) \ge N $ charge,
which it can then distribute to its $N$ extensions into embeddings of $S$.

To complete the discharging argument,
it remains to bound the total amount of charge that each embedding of a $k$-vertex tree non-isomorphic to $S$ sends.
If the embedding is isomorphic to $S'$, then $v_{S\to T}$ is the unique vertex that is the root of two forks.
If the embedding is isomorphic to $S''$, then $v_{S\to T}$ is either the second or second-last internal vertex.
If the embedding is isomorphic to $S'''$, then $v_{S\to T}$ is either the first or last internal vertex.
Finally, if the embedding is isomorphic to $S_Q$,
then there are at most two choices of edges that could have been added as part of $Q$, and
once these edges are chosen and removed,
$v_{S\to T}$ is either the first, second, second-last, or last internal vertex of the resulting caterpillar.
When the edge(s) added from $Q$ are removed from the embedding and $v_{S\to T}$ is chosen,
the edges of the stub can be recovered by removing the forks and leaves rooted at $v_{S\to T}$ from the embedding.
Hence, the embedding sends at most $54 \cdot (\max\{1,2,2\} + 2 \cdot 4) \le 999$ charge, and
the density of $S$ in $T$ is at most $1 - 10^{-3}$.
\end{proof}

The next lemma deals with caterpillars $S$ that are not covered by Theorem~\ref{thm:forks} and Lemma~\ref{lem:cater2}.

\begin{lemma}
\label{lem:cater3+4}
Let $S$ be a caterpillar with $k\ge 10$ vertices that is not a path such that
the path $v_1,\dots,v_m$ formed by its internal vertices satisfies either $m=2$, or $m\ge 3$ and
the degrees of $v_2$ and $v_{m-1}$ equal $2$.
If $T$ is a tree with radius at least $4k$, then $d(S,T) \leq 1-10^{-3}$.
\end{lemma}

\begin{proof}
Let $\alpha > 0$ and $\beta >0$ be the number of leaves adjacent to $v_1$ and $v_m$ respectively. By symmetry, we can assume that $\alpha \leq \beta$.
Let $S'$ be the caterpillar obtained from $S$ by removing the $\alpha$ leaves adjacent to $v_1$ and $\beta$ leaves adjacent to $v_m$.

Let $v_T$ be a vertex of $T$ such that
there exist $(2k+1)$-vertex paths $P_1$ and $P_2$ starting at $v_T$ that
are disjoint except at $v_T$ itself; such a choice
is possible because the radius of $T$ is at least $4k$.

In this proof, a \emph{stub} is an embedding of $S'$ in $T$ together with a choice of orientation for the longest path in the embedding.
(The length of this path is the same as the distance between $v_1$ and $v_m$.)
Given a stub, let $v'_1, \dots, v'_m$ be the vertices of the longest path in the embedding, ordered according to the chosen orientation,
and let $A$ and $B$ be the degrees of $v'_1$ and $v'_m$ minus 1.
Let $Q$ be the unique path in $T$ between the vertices $v_{S\to T}$ and $v_T$
prolonged by $P_1$ if $P_1$ does not contain the edge from $v_T$ towards $v_{S\to T}$ and prolonged by $P_2$ otherwise.

We analyze the case when $\alpha=\beta=1$ separately at the end of the proof,
so for now suppose that $\beta\geq 2$.
The number of ways the embedding of $S'$ can be extended to an embedding of $S$ with each $v'_i$ corresponding to $v_i$ is $\binom{A}{\alpha}\binom{B}{\beta}$.
We will associate to each embedding of $S'$ in $T$ a set of $N$ embeddings of $k$-vertex trees non-isomorphic to $S$ so that
\[\binom{A}{\alpha}\binom{B}{\beta}\le 9N.\]
Note that if $A<\alpha$ or $B<\beta$, then there is nothing to prove,
so we assume that $A\ge\alpha$ and $B\ge\beta$.

If $B=\beta$,
then we consider extensions obtained from the embedding of $S'$
by adding $\alpha$ leaves to $v'_1$, $\beta-2$ leaves to $v'_m$ and
then the first two edges of $Q$ not contained in the embedding.
If the obtained embedding is a caterpillar, then the sum of the degrees of its first and last internal vertices is less than the sum of the degrees of the first and last internal vertices of $S$. Therefore, the $N=\binom{A}{a} = \binom{A}{\alpha}\binom{B}{\beta}$ obtained embeddings are not isomorphic to $S$.

If $A=\alpha\ge 2$,
then consider extensions obtained from the embedding of $S'$
by adding $\alpha-2$ leaves to $v'_1$, $\beta$ leaves to $v'_m$ and
then the first two edges of $Q$ not contained in the embedding.
Again, if the obtained embedding is a caterpillar, then the sum of the degrees of its first and last internal vertices is less than the sum of the degrees of the first and last internal vertices of $S$.
Therefore, the obtained embeddings are not isomorphic to $S$, and
their number is $N=\binom{B}{\beta}$, which is equal to $\binom{A}{\alpha}\binom{B}{\beta}$.

If $A=\alpha=1$ but $2\le\beta<B$,
then the number of extensions to $S$ is $\binom{B}{\beta}$, and
we consider extensions of $S'$ obtained
by either adding $\beta+1$ leaves to $v'_2$ or
adding $\beta-1$ leaves to $v'_2$ and the first two edges of $Q$ not contained in the embedding;
the number $N$ such extensions is
\[N=\binom{B}{\beta-1}+\binom{B}{\beta+1}\ge\binom{B}{\beta}.\]

Hence, we can assume that $A>\alpha$ and $B>\beta$ in the remainder of the analysis of the case $\beta\ge 2$.
We first deal with the case when $\alpha\not=\beta-1$.
Consider extensions obtained from the embedding of $S'$
by either adding $\alpha+1$ leaves to $v'_1$ and $\beta-1$ leaves to $v'_m$ or
adding $\alpha-1$ leaves to $v'_1$ and $\beta+1$ leaves to $v'_m$;
the number of such extensions is
\begin{align*}
N & =\binom{A}{\alpha+1}\binom{B}{\beta-1}+\binom{A}{\alpha-1}\binom{B}{\beta+1} \\
  & =\binom{A}{\alpha-1}\binom{B}{\beta-1}\left(\frac{(A-\alpha+1)(A-\alpha)}{\alpha(\alpha+1)}+\frac{(B-\beta+1)(B-\beta)}{\beta(\beta+1)}\right) \\
  & \ge\frac{1}{4}\binom{A}{\alpha-1}\binom{B}{\beta-1}\left(\frac{(A-\alpha+1)^2}{\alpha^2}+\frac{(B-\beta+1)^2}{\beta^2}\right) \\
  & \ge\frac{1}{2}\binom{A}{\alpha-1}\binom{B}{\beta-1}\frac{A-\alpha+1}{\alpha}\frac{B-\beta+1}{\beta}=\frac{1}{2}\binom{A}{\alpha}\binom{B}{\beta}.
\end{align*}

It remains to analyze the case $\alpha=\beta-1$.
Consider extensions obtained from the embedding of $S'$
by either adding $\alpha+2$ leaves to $v'_1$ and $\beta-2$ leaves to $v'_m$ or
adding $\alpha-1$ leaves to $v'_1$ and $\beta+1$ leaves to $v'_m$; the number of such extensions is
\begin{align*}
N & =\binom{A}{\alpha+2}\binom{B}{\beta-2}+\binom{A}{\alpha-1}\binom{B}{\beta+1},
\end{align*}
unless $A=\alpha+1$.
We next argue that $\binom{A}{\alpha}\binom{B}{\beta}\le 9N$.
If $\binom{A}{\alpha-1}\binom{B}{\beta+1}\le \frac{1}{9}\binom{A}{\alpha}\binom{B}{\beta}$,
then $9\le\frac{(\beta+1)(A-\alpha+1)}{\alpha(B-\beta)}$
and
\begin{align*}
\binom{A}{\alpha+2}\binom{B}{\beta-2}
  & = \frac{\beta(\beta-1)(A-\alpha)(A-\alpha-1)}{(\alpha+2)(\alpha+1)(B-\beta+2)(B-\beta+1)}\binom{A}{\alpha}\binom{B}{\beta} \\
  & \ge \frac{2^2}{3^4\cdot 6^2}\left(\frac{(\beta+1)(A-\alpha+1)}{\alpha(B-\beta)}\right)^2\binom{A}{\alpha}\binom{B}{\beta} \\
  & \ge \frac{9^2}{3^6}\binom{A}{\alpha}\binom{B}{\beta}=\frac{1}{9}\binom{A}{\alpha}\binom{B}{\beta}\,,
\end{align*}
so $N\geq \frac{1}{9}\binom{A}{\alpha}\binom{B}{\beta}$.

Finally, we deal with the case $A=\alpha+1$.
Consider extensions obtained from the embedding of $S'$
by adding $\alpha-1$ leaves to $v'_1$, $\beta-1$ leaves to $v'_m$ and
then the first two edges of $Q$ not contained in the embedding, or
adding $\alpha-1$ leaves to $v'_1$ and $\beta+1$ leaves to $v'_m$;
as before, these extensions are not isomorphic to $S$, and
the number of such extensions is
\[\binom{A}{\alpha-1}\binom{B}{\beta-1}+\binom{A}{\alpha-1}\binom{B}{\beta+1}\ge\binom{A}{\alpha+1}\binom{B}{\beta-1}+\binom{A}{\alpha-1}\binom{B}{\beta+1},\]
which is at least $\binom{A}{\alpha}\binom{B}{\beta}/2$ as established in the case $\alpha\not=\beta-1$.

We complete the case $\beta\ge 2$
by a discharging procedure similar to that used in the proof of Lemma \ref{lem:cater2}.
Each embedding of a $k$-vertex tree non-isomorphic to $S$ sends a charge of $9$ to each stub that it extends.
In this way, each stub receives $ 9 N \ge \binom{A}{\alpha}\binom{B}{\beta}$ charge,
which is then distributed to the extensions of the stub into embeddings of $S$.

To finish the analysis of the discharging argument,
it remains to bound the total amount of charge that each embedding of a $k$-vertex tree non-isomorphic to $S$ sends.
Fix an embedding of a $k$-vertex tree non-isomorphic to $S$ and
suppose that it can be obtained from a stub by at least one of the above processes.
If this process \emph{does not} involve adding two unused edges of $Q$,
then the embedding is a caterpillar and the vertex that corresponds to $v'_1$
is either the first or last internal vertex of the embedding, or
the unique leaf of the first internal vertex if the first internal vertex has degree $2$, or
the unique leaf of the last internal vertex if the last internal vertex has degree $2$.
For each choice of $v'_1$ there is at most one valid choice of $v'_m$ at the correct distance in $T$, and
the stub is then determined by removing the leaves of the embedding adjacent to $v'_1$ and $v'_m$.
Thus, the embedding sends at most $9 \cdot 4$ units of charge in this way. If the embedding can be obtained from a stub by a process that
\emph{does} involve adding two unused edges of $Q$,
then there are at most two places where these edges could have been added.
After choosing which of these sets of two edges to remove from the embedding,
we follow the same procedure as above to obtain the possible stubs,
so the embedding sends at most $9 \cdot 2 \cdot 4$ additional charge in this way.
In total,
each embedding of a $k$-vertex tree non-isomorphic to $S$ sends at most $9\cdot 3\cdot 4 < 999$ units of charge,
so the density of $S$ in $T$ is at most $1-10^{-3}$.

Finally, we return to the case $\alpha=\beta=1$;
the argument is again based on a discharging procedure.
Since $S$ is not a path, we have $m\geq 5$.
Consider extensions of the embedding of $S'$ obtained
by either adding two leaves to the vertex $v'_1$ or adding two leaves to the vertex $v'_m$.
The number of such extensions is $\binom{A}{2}+\binom{B}{2}\ge \frac{1}{4}AB$ unless $A=B=1$.
If $A=B=1$, then $\binom{A}{\alpha}\binom{B}{\beta}=1$ and we
consider the embedding obtained from $S'$ by adding the first two edges of $Q$ that are not contained in $S'$.
If the first of these two edges attaches to a vertex of $S'$ other than $v'_1$, $v'_2$, $v'_{m-1}$, or $v'_m$,
then the obtained embedding is not a caterpillar.
If it attaches to $v'_2$ or $v'_{m-1}$, then the second or second-last internal vertex of the obtained caterpillar has degree $3$.
Thus the obtained embedding is not isomorphic to $S$ unless the added edges attach to $v'_1$ or $v'_m$.
By symmetry, we assume that they attach to $v'_m$.
Any isomorphism between $S$ and the obtained embedding maps the vertex $v_{m-1}$ of $S$ to the vertex $v'_3$ of the obtained embedding.
In particular, if the two trees are isomorphic, then the degree of $v'_3$ in $S'$ is $2$, and
we instead consider the embedding obtained from $S'$ by adding a leaf to $v'_2$ and a leaf to $v'_m$, if such an embedding exists.
If such an embedding does not exist,
we instead consider the embedding obtained from $S'$ by removing the edge $v'_1 v'_2$ and
adding the first three edges of $Q$.
In all cases above from the analysis of the case $A=B=1$,
the obtaining embedding is non-isomorphic to $S$.

Thus, when each embedding of a $k$-vertex tree non-isomorphic to $S$
sends $4$ units of charge to each stub that it extends,
each stub such that $A\not=1$ or $B\not=1$ receives at least $AB$ units of charge,
which it can then redistribute to its extensions into embeddings of $S$.
We next estimate additional charge because of embeddings considered in the case $A=B=1$;
for this analysis, fix an embedding of a $k$-vertex tree non-isomorphic to $S$.
\begin{itemize}
\item If the embedding can be obtained from the process of adding the first two unused edges of $Q$,
      then there are at most two places where these edges could have been added.
      After choosing which of these two choices of two edges to remove,
      the vertices $v'_1$ and $v'_m$ must correspond to the unique leaves of the first and last internal vertices, and
      the stub is determined once a choice for the correspondence is made.
\item If the embedding can be obtained from the process of adding a leaf to $v'_2$ and $v'_{m-1}$,
      then either the first or last internal vertex of the embedding has two leaves, and
      this vertex corresponds to either $v'_2$ or $v'_{m-1}$.
      One of the two leaves then corresponds to either $v'_1$ or $v'_m$, and
      the stub is determined by choosing the correspondence, removing the other leaf, and
      removing the unique leaf of the terminal internal vertex at the other end of the caterpillar. 
\item If the embedding can be obtained from the process of removing the edge $v'_1 v'_2$ and
      adding the first three unused edges of $Q$,
      there are at most two places where these edges could have been added.
      After choosing which of these two options of three edges to remove,
      $v'_3$ is the first or last internal vertex of the resulting caterpillar,
      $v'_2$ is the unique leaf of $v'_3$ in the caterpillar, and
      $v'_1$ is the unique neighbor of $v'_2$ in $T$ that is not in the caterpillar.
\end{itemize}      
Each embedding of a $k$-vertex tree non-isomorphic to $S$ sends at most $12$ additional units of charge and
so it sends at most $16<99$ units of charge in total.
We conclude that if $\alpha = \beta = 1$, then $d(S,T) \leq 1-10^{-2}$.
\end{proof}

The next theorem summarizes our analysis of caterpillars.

\begin{theorem}
\label{thm:cater}
Every caterpillar $S$ with $|S|\geq 17$ that is neither a star nor a path has inducibility at most $1-10^{-4}$.
\end{theorem}

\begin{proof}
The limit density of $S$ in any sequence of trees with bounded radius is $0$ by Proposition~\ref{prop:radius},
so it remains to investigate the density of $S$ in trees with unbounded radius.
Since $S$ is not a star, $S$ has more than one internal vertex.
The case where $S$ has two internal vertices is covered by Lemma~\ref{lem:cater3+4},
the case where $S$ has three internal vertices is covered by Lemma~\ref{lem:cater3+4} and
the third case of Theorem~\ref{thm:forks} (depending on whether the middle internal vertex has degree 2), and
the case where $S$ has four or more internal vertices is covered by Lemma~\ref{lem:cater2}, Lemma~\ref{lem:cater3+4}, and the second case of Theorem~\ref{thm:forks} (depending on the order of forks rooted at the second or second-last internal vertex of $S$ and their degrees).
\end{proof}

\section{Inducibility bounded away from 1}
\label{sec:main}

As detailed in the proof of Theorem~\ref{thm:main} below, the results in the preceding sections show that the inducibility of every tree with at least $17$ vertices is at most $1-10^{-8}$. On the other hand, the inducibility of every $k$-vertex tree $X$ that is neither a path nor a star is less than $1$, since for every convergent sequence of trees $(T_n)_{n\in\NN}$ with $|T_n|\to\infty$, if $\lim_{n\to\infty}d(X,T_n)=1$, then $\lim_{n\to\infty}d(P_k,T_n)=\lim_{n\to\infty}d(S_k,T_n)=0$, contradicting 
Proposition~\ref{prop:plimit}. In particular, trees with at most $16$ vertices have inducibility bounded away from $1$. 
These two results combined imply that the inducibility of every tree is at most $1-\varepsilon$ for some fixed constant $\varepsilon > 0$. 
In the interest of obtaining an explicit value of $\varepsilon$, we provide a crude upper bound on the inducibility of small trees.

\begin{lemma}
\label{lem:elimit}
For $k\geq 5$, the inducibility of every $k$-vertex tree $S$ that is neither a path nor a star is at most $1-k^{-(2k-3)}$.
\end{lemma}

\begin{proof}
Since the inducibility of $S$ is defined with respect to $d(S,T)$ where $|T|\to\infty$, it suffices to show that $d(S,T)\le 1-k^{-(2k-3)}$ for every tree $T$ with $|T|\geq k^k$. We prove the bound by a discharging argument. Every embedding of a $k$-vertex tree in $T$ that is not isomorphic to $S$ begins with one unit of charge and distributes the charge according to the following rules.

Every embedding $S''$ of a $k$-vertex tree that is neither isomorphic to $S$ nor a star
distributes its charge equally among all embeddings $S'$ of $k$-vertex trees that
share an edge with $S''$ such that the maximum degree of a vertex of $S'$ in $T$ is at most $k$.
Since there are at most $(k-1)!(k-1)^{k-2}\le k^{2k-3}$ such embeddings $S'$ for every embedding $S''$,
each embedding $S'$ of a $k$-vertex tree such that the maximum degree of a vertex of $S'$ in $T$ is at most $k$
receives at least $k^{-(2k-3)}$ units of charge.

Every embedding $S''$ of a $k$-vertex star distributes its charge as follows.
Let $v$ be the center of $S''$ and $d$ its degree in $T$.
The embedding $S''$ distributes its charge equally among all embeddings $S'$ of $k$-vertex trees that
share an edge with $S''$ such that the maximum degree of the vertices of $S'$ in $T$ is at most $d$.
Each such embedding $S'$ receives charge from at least $\binom{d-1}{k-2}$ embeddings of stars and
each embedding of a $k$-vertex star centered at $v$ sends charge to at most $(k-1)!k(d-1)^{k-2}\le k!d^{k-2}$ embeddings of $S$.
Hence, each embedding $S'$ of a $k$-vertex tree such that the maximum degree of the vertices of $S'$ in $T$ is $d>k$
receives at least
\[\binom{d-1}{k-2}\frac{1}{k!d^{k-2}}\ge\frac{1}{k!k^{k-2}}\ge k^{-(2k-3)}\]
units of charge.

Since $d(S,T)<1$ and every embedding of a $k$-vertex tree in $T$ (regardless of whether the embedding is of $S$ or not)
has at least $k^{-(2k-3)}$ units of charge at the end of the process described above,
it follows that $d(S,T)\le 1-k^{-(2k-3)}$.
\end{proof}

We now combine the results of Sections~\ref{sec:major}--\ref{sec:main} to prove the first main result of this paper.
As mentioned earlier, we do not attempt to optimize the upper bound on the inducibility presented in the theorem.

\begin{theorem}
\label{thm:main}
The inducibility of every tree $S$ that is neither a star nor a path is at most $1-10^{-35}$.
\end{theorem}

\begin{proof}
Since $S$ is neither a star nor a path, $|S|\geq 5$. If $|S|\leq 16$, then the inducibility of $S$ is at most $1-16^{-29} \leq 1-10^{-35}$ by Lemma \ref{lem:elimit}.
Now assume that $|S| \geq 17$. If $S$ is a caterpillar, then the inducibility of $S$ is at most $1-10^{-4}$ by Theorem~\ref{thm:cater}.
If $S$ is not a caterpillar, then consider an arbitrary hub $v$ of $S$. By the definition of a hub, $v$ is the root of three non-trivial branches.
If $v$ is adjacent to at most one leaf, or is the root of at most one fork and at least three major branches, then the inducibility of $S$ is at most $1-10^{-7}$ by Theorem~\ref{thm:major}.
Otherwise, $v$ is adjacent to at least two leaves, and additionally is either the root of at least two forks or at most two major branches.
In either case, $v$ is the root of a fork since every non-trivial branch is either a fork or is major.
Theorem~\ref{thm:forks} then guarantees that the inducibility of $S$ is at most $1-10^{-4}$.
\end{proof}

\section{Inducibility bounded away from 0}
\label{sec:sparkler}

A \emph{sparkler} is a graph obtained from a star by subdividing one of its edges once. The following result shows that sparklers are an infinite class of trees with inducibility bounded away from $0$, thus answering Problem~4 of Bubeck and Linial~\cite{BubL16} in the affirmative.

\begin{theorem}
\label{thm:sparkler}
The inducibility of every sparkler with at least four edges is at least $13/165$. 
\end{theorem}

\begin{proof}
Fix $k\geq 4$, and let $S'_k$ be the sparkler with $k$ edges;
that is, the graph obtained from the star with $k-1$ leaves by subdividing one of its edges.
We will construct a sequence $(T_n)_{n\in\NN}$ of trees with $|T_n|\to \infty$ such that $d(S'_k,T_n)\ge 13/165$,
which implies the theorem.

As illustrated in Figure~\ref{fig:sparkler}, let $T_n$ be the tree
obtained from a path with $n(k+1)+k$ vertices (called the \emph{spine})
by adding $3k$ leaves to its $(j(k+1))$-th vertex for $j\in\{1,\dots,n\}$;
each of the $n$ vertices to which the leaves are attached is called a \emph{vertebra}.

\begin{figure}
\begin{center}
\epsfbox{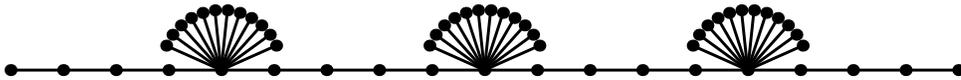}
\end{center}
\caption{The tree $T_3$ constructed for $k=4$ in the proof of Theorem~\ref{thm:sparkler}.}
\label{fig:sparkler}
\end{figure}

Observe that the number of copies of $S'_k$ in $T_n$ is
\[ 2n \binom{3k+1}{k-2}. \]
We next count the number of all $k$-edge subtrees of $T_n$.
Each $k$-edge subtree of $T_n$ contains exactly one of the vertebrae.
The number of $k$-edge subtrees that contain exactly $j$ edges from the spine of $T_n$ for $j\in\{0,\dots,k\}$ is
\[(j+1)n\binom{3k}{k-j}.\]
Thus the total number of $k$-edge subtrees of $T_n$ is
\begin{equation}
n\sum_{j=0}^k (j+1)\binom{3k}{k-j}.\label{eq:sparkler1}
\end{equation}
Observe that
\[\frac{\binom{3k}{k-j}}{\binom{3k}{k-j-1}}=\frac{2k+j+1}{k-j}\ge 2\]
for every $j\in\{0,\dots,k-1\}$,
which can be used iteratively on \eqref{eq:sparkler1} to bound the number of $k$-edge subtrees of $T_n$:
\[n\sum_{j=0}^k (j+1)\binom{3k}{k-j}
  \leq\binom{3k}{k}n\sum_{j=0}^k\frac{j+1}{2^j}
  \leq\binom{3k}{k}n\sum_{j=0}^\infty\frac{j+1}{2^j}\,.\]
The arithmetico--geometric series in the last expression sums to $4$, so it follows that the density of $S'_k$ in $T_n$ is at least
\[\frac{2\binom{3k+1}{k-2}n}{4\binom{3k}{k}n}=\frac{(3k+1)k(k-1)}{2(2k+3)(2k+2)(2k+1)}\ge\frac{13}{165},\]
where the last inequality holds since $k\ge 4$.
\end{proof}

We remark that the construction from Theorem~\ref{thm:sparkler} can be optimized
by adding $\lceil\alpha k\rceil$ leaves instead of adding $3k$ to the vertebrae for $\alpha\approx 2.8507$,
which yields that the inducibility of sufficiently large sparklers is at least $0.19004$,
while the bound presented in the proof converges to $3/16$ for $k$ tending to infinity.

\section{Universal sequence of trees}
\label{sec:universal}

In this section, we prove the existence of a universal sequence of trees.

\begin{theorem}
\label{thm:universal}
There exists a sequence $(T_n)_{n\in\NN}$ of trees
in which the limit density $\lim\limits_{n\to \infty} d(S,T_n)$ of every tree $S$ is positive.
\end{theorem}

\begin{proof}
To describe the construction, we first define a gluing operation on trees, which we denote by $\oplus$;
this operation has already been used in the context of tree profiles in~\cite{BubL16}.
If $T$ and $T'$ are trees, then 
$T\oplus T'$ is any tree obtained from the disjoint union of $T$ and $T'$
by joining a vertex of $T$ and a vertex of $T'$ by an edge. The resulting tree depends, of course, on which vertices are chosen to be joined by an edge, but the choice will not influence our arguments as
long as the maximum degree of the resulting tree is controlled when we do a sequence of these operations. In particular, if we always choose a leaf of $T$ and a leaf of $T'$, then the maximum degree does not increase (unless $T\cong K_2$ or $T'\cong K_2$).

Observe that if $\Delta(T\oplus T')$ is the maximum degree of the resulting tree,
then the number of $k$-vertex trees containing the gluing edge
is at most $(k-1)^{k-1}\left(\Delta(T\oplus T')-1\right)^{k-1}\le\left(k(\Delta(T\oplus T')-1)\right)^{k-1}$ (start
with the gluing edge and then add $k-1$ edges iteratively, having at most $(k-1)\left(\Delta(T\oplus T')-1\right)$ at each iteration),
which yields that
\begin{equation}
Z_k(A) + Z_k(B) \leq Z_k(A \oplus B) \leq Z_k(A) + Z_k(B) + \left(k(\Delta(T\oplus T')-1)\right)^{k-1}
\label{ineq:glue1}
\end{equation}
We further define an iterative version of the gluing operation $\oplus$
by setting $T^{\oplus 1}=T$ and $T^{\oplus \ell} = T^{\oplus (\ell - 1)}\oplus T$ for $\ell\ge 2$.

Let $B_d$ be the complete $d$-ary tree of depth $d$; that is, $B_d$ is the rooted tree such that every internal vertex has $d$ children and
every leaf is at distance $d$ from the root. Observe that
$|B_d|=1+d+d^2+\cdots+d^d=\frac{d^{d+1}-1}{d-1}\le d^{d+1}$,
the maximum degree of $B_d$ is $d+1$ if $d\geq 2$ and $1$ if $d=1$, and
the tree $B_d$ contains a copy of every tree with $d$ vertices.
We now define the sequence $(T_n)_{n\in\NN}$ in the statement of the theorem.
The tree $T_n$ is obtained by gluing copies of the trees $B_1,\ldots,B_n$ in a ratio such that 
a significant proportion of the $k$-vertex subtrees in the resulting tree $T_n$ arises from copies of $B_1,\ldots,B_k$.
Formally, set $T_1=B_1$, and for $n\ge 2$, define 
$$T_n=B_n\oplus\left(T_{n-1}^{\oplus n^2}\right),$$  where the gluing operation is performed so that $\Delta(T_n)\leq n+1$. 
Observe that $T_n$ consists of $\left(\frac{n!}{d!}\right)^2$ copies of $B_d$ for $d\in\{1,\dots,n\}$. See Figure~\ref{fig:universal} for an illustration.

\begin{figure}
\begin{center}
\epsfbox{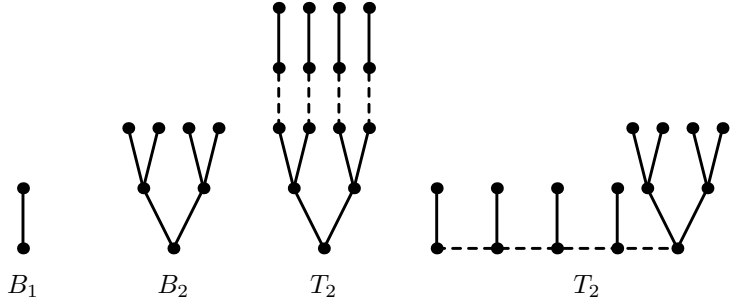}
\end{center}
\caption{The trees $B_1$, $B_2$, and two possible choices of $T_2$. Edges added by the operation $\oplus$ are dashed.}
\label{fig:universal}
\end{figure}

Fix a $k$-vertex tree $S$ with $k\ge 3$ for the rest of the proof, and note that
\begin{equation}
d(S,T_n)\ge\frac{1}{Z_k(T_n)}\cdot\left(\frac{n!}{k!}\right)^2
\label{ineq:glueZ}
\end{equation}
for every $n\ge k$.
We next upper bound the number of $k$-vertex subtrees in $T_n$ using \eqref{ineq:glue1}:
\begin{align*}
Z_k(T_n) & \le Z_k(B_n) + Z_k(T_{n-1}^{\oplus n^2}) + (kn)^{k-1}  \\
& \le Z_k(B_n) + n^2 Z_k(T_{n-1}) + n^2(kn)^{k-1} \\
         & = Z_k(B_n) + n^2 Z_k(T_{n-1}) + k^{k-1}n^{k+1}.
\end{align*}
Iterating the inequality, we obtain that
\[Z_k(T_n)\le \left[ \sum_{d=k+1}^n\left(\frac{n!}{d!}\right)^2\left(Z_k(B_d)+k^{k-1}d^{k+1}\right) \right]+
             \left(\frac{n!}{k!}\right)^2 Z_k(T_k).\]
We next analyze the sum from the above expression:
\begin{align*}
\lim_{n\to\infty}\frac{\sum\limits_{d=k+1}^n\left(\frac{n!}{d!}\right)^2\left(Z_k(B_d)+k^{k-1}d^{k+1}\right)}{\left(\frac{n!}{k!}\right)^2}
  & = \sum_{d=k+1}^{\infty}\left(\frac{k!}{d!}\right)^2\left(Z_k(B_d)+k^{k-1}d^{k+1}\right)\\
  & \le \sum_{d=k+1}^{\infty}\left(\frac{k!}{d!}\right)^2\left(d^{d+1}k^{k-1}d^{k-1}+k^{k-1}d^{k+1}\right)\\
  & \le (k!)^2k^{k-1}\sum_{d=k+1}^{\infty}\frac{2d^{d+k}}{(d!)^2}\\
  & \le 2(k!)^2k^{k-1}\sum_{d=k+1}^{\infty}\frac{d^{d+k}e^{2d-2}}{d^{2d}}\\
  & \le 2(k!)^2k^{k-1}e^{2k}.
\end{align*}
This combines with \eqref{ineq:glueZ} to imply that
\[d(S,T_n)\ge\frac{1}{2(k!)^2k^{k-1}e^{2k}+Z_k(T_k)}>0.\]
Considering a convergent subsequence of $(T_n)$ if necessary,
we deduce that there exists a convergent sequence of trees in which the limit density of every tree $S$ is positive.
\end{proof}

We remark that the choice of the vertices for the gluing operation permits creating sequences of trees with different ``shapes''. For example, the trees can be grown to the depth as the left tree $T_2$ in Figure~\ref{fig:universal} or
along a path as the right tree $T_2$ in Figure~\ref{fig:universal}.

\let\oldthebibliography=\thebibliography
\let\endoldthebibliography=\endthebibliography
\renewenvironment{thebibliography}[1]{%
\begin{oldthebibliography}{#1}%
	\setlength{\parskip}{0ex}%
	\setlength{\itemsep}{0ex}%
}{\end{oldthebibliography}}

\bibliographystyle{bibstyle}
\bibliography{tprofile}

\end{document}